\documentclass[12pt,a4paper]{amsart}

\usepackage{amsmath,amsfonts,amssymb}

\usepackage[hmargin=2cm,vmargin=2cm]{geometry}

\usepackage{hyperref}
\usepackage{nameref,zref-xr}

\newcommand{\mbC}{\mathbb C}

\newcommand{\oM}{\overline{\mathcal M}}

\def\cM{{\mathcal{M}}}
\def\oM{{\overline{\mathcal{M}}}}

\def\mbQ{{\mathbb Q}}
\def\d{{\partial}}

\newcommand{\<}{\left<}
\renewcommand{\>}{\right>}

\newcommand{\DR}{\mathrm{DR}}

\newcommand{\Coef}{\mathrm{Coef}}

\newcommand{\mcF}{\mathcal F}
\newcommand{\mbR}{\mathbb R}
\newcommand{\mcG}{\mathcal G}
\newcommand{\ob}{\overline{b}}

\newtheorem{theorem}{Theorem}[section]
\newtheorem{proposition}[theorem]{Proposition}
\newtheorem{lemma}[theorem]{Lemma}

\theoremstyle{definition}

\numberwithin{equation}{section}

\title[DR cycles and the $n$-point function]{Double ramification cycles and the $n$-point function for the moduli space of curves}

\author{Alexandr Buryak}
\address{Department of Mathematics, ETH Z\"urich, Switzerland, and \newline
\indent Faculty of Mechanics and Mathematics, Lomonosov Moscow State University, Russian Federation}
\email{buryaksh@gmail.com}



\thanks{The author was supported by grant Russian Science Foundation N16-11-10260, project "Geometry and mathematical physics of integrable systems". We are grateful to R. Pandharipande and S. Shadrin for useful discussions and to an anonymous referee for a number of suggestions that helped us to improve the exposition of the paper.
}

\begin{document}

\begin{abstract}
In this paper, using the formula for the integrals of the $\psi$-classes over the double ramification cycles found by S.~Shadrin, L.~Spitz, D.~Zvonkine and the author, we derive a new explicit formula for the $n$-point function of the intersection numbers on the moduli space of curves.
\end{abstract}

\maketitle

\section{Introduction}

Consider the moduli space~$\oM_{g,n}$ of stable complex algebraic curves of genus~$g$ with~$n$ marked points. The class $\psi_i\in H^2(\oM_{g,n},\mbQ)$ is defined as the first Chern class of the line bundle over~$\oM_{g,n}$ formed by the cotangent lines at the $i$-th marked point. We define the intersection numbers by
\begin{gather}\label{eq:intersection numbers}
\<\tau_{d_1}\tau_{d_2} \cdots \tau_{d_n}\>_g:=\int_{\oM_{g,n}}\psi_1^{d_1}\psi_2^{d_2}\cdots\psi_n^{d_n}.
\end{gather}
In the unstable case $2g-2+n\le 0$ we define the bracket to be equal to zero. The bracket~\eqref{eq:intersection numbers} vanishes unless the dimension constraint 
\begin{gather*}
\sum_{i=1}^n d_i=3g-3+n
\end{gather*}
is satisfied. Introduce variables $t_0,t_1,t_2,\ldots$. The celebrated conjecture of E.~Witten~\cite{Wit91}, proved by M.~Kontsevich~\cite{Kon92}, says that for the following generating function
$$
F(t_0,t_1,\ldots):=\sum_{\substack{g\ge 0\\n\ge 1}}\sum_{d_1,\ldots,d_n\ge 0}\<\tau_{d_1}\cdots\tau_{d_n}\>_g\frac{t_{d_1}\cdots t_{d_n}}{n!},
$$
the exponent $e^F$ is a tau-function of the KdV hierarchy in the variables $T_{2i+1}=\frac{t_i}{(2i+1)!!}$.

In this paper we discuss a different generating function for the intersection numbers~\eqref{eq:intersection numbers}. Let $n\ge 1$. Introduce variables $x_1,\ldots,x_n$. Define 
$$
\mcF(x_1,\ldots,x_n):=\sum_{g\ge 0}\mcF_g(x_1,\ldots,x_n),
$$
where
\begin{gather*}
\mcF_g(x_1,\ldots,x_n):=\sum_{d_1,\ldots,d_n\ge 0}\<\tau_{d_1}\cdots\tau_{d_n}\>_g x_1^{d_1}\cdots x_n^{d_n}.
\end{gather*}
The function $\mcF(x_1,\ldots,x_n)$ is often called the $n$-point function. There are several known closed formulas for it~\cite{BDY15,BH07,Oko02}. In this paper we derive a simple new formula for~$\mcF(x_1,\ldots,x_n)$ using the result of~\cite{BSSZ15}. We believe that our approach can be useful for the understanding of the structure of more general Gromov-Witten invariants.

Let us formulate our main result. Introduce variables $a_1,a_2,\ldots,a_n$. We will use the following notations.
\begin{itemize}
\item Let $\zeta(x):=e^{\frac{x}{2}}-e^{-\frac{x}{2}}$.
\item For a permuation $\sigma\in S_n$ denote $a_i':=a_{\sigma(i)}$ and $x_i':=x_{\sigma(i)}$.
\item Finally,
$$
\begin{array}{|cc|}
a & b
\\ c & d
\end{array}
= ad-bc.
$$
\end{itemize}
We define the functions $P_n(a_1,\ldots,a_n;x_1,\ldots,x_n)$ by
\begin{align}
&P_1(a_1;x_1):=\frac{1}{x_1},\notag\\
&P_n(a_1,\ldots,a_n;x_1,\ldots,x_n):=\label{eq:main definition}\\
=&\sum_{\substack{\sigma\in S_n\\ \sigma(1) = 1}}
x'_2\cdots x'_{n-1}
\frac{
\zeta\left(
\begin{array}{|cc|}
a_1' & a_2'
\\ x_1' & x_2'
\end{array}
\right)
\zeta\left(
\begin{array}{|cc|}
a_1'+a_2' & a_3'
\\ x_1'+x_2' & x_3'
\end{array}
\right)
\cdots
\zeta\left(
\begin{array}{|cc|}
a_1' + \cdots + a_{n-1}' & a_n'
\\ x_1' + \cdots + x_{n-1}' & x_n'
\end{array}
\right)}
{\begin{array}{|cc|}
a_1' & a_2'
\\ x_1' & x_2'
\end{array}
\; \; 
\begin{array}{|cc|}
a_2' & a_3'
\\ x_2' & x_3'
\end{array}
\;
\cdots
\;
\begin{array}{|cc|}
a_{n-1}' & a_n'
\\ x_{n-1}' & x_n'
\end{array}
},\quad n\ge 2\notag.
\end{align}
At first sight it appears that for $n\ge 2$ the function $P_n(a;x)$ has simple poles along the hyperplanes $\{a_ix_j - a_j x_i=0\}$, but in~\cite[Remark 1.6]{BSSZ15} it was shown that for $n\ge 2$ the function~$P_n(a;x)$ is a power series in $x_1,\ldots,x_n$ with coefficients in $\mbC[a_1,\ldots,a_n]$. Moreover, by~\cite[Remark 1.5]{BSSZ15}, the function $P_n(a;x)$ is symmetric with respect to simultaneous permutations of the variables $a_1,\ldots,a_n$ and the variables $x_1,\ldots,x_n$. Our main result is the following theorem.

\begin{theorem}\label{theorem:main}
We have
\begin{multline}
\mcF(x_1,\ldots,x_n)=\label{eq:main formula}\\
=\frac{e^{\frac{\left(\sum x_i\right)^3}{24}}}{\left(\sum x_i\right)\prod\sqrt{2\pi x_i}}\int_{\mbR^n}e^{-\sum\frac{a_i^2}{2x_i}}P_n(\sqrt{-1}a_1,\ldots,\sqrt{-1}a_n;x_1,\ldots,x_n)da-\frac{\delta_{n,1}}{x_1^2}-\frac{\delta_{n,2}}{x_1+x_2},
\end{multline}
where $da:=da_1\cdots da_n$.
\end{theorem}
Let us make a few comments about the right-hand side of equation~\eqref{eq:main formula} and, in particular, show that it is a power series in $x_1,\ldots,x_n$. Note that for $x>0$ and $d\ge 0$ we have
\begin{gather}\label{eq:Gauss integral}
\frac{1}{\sqrt{2\pi x}}\int_\mbR a^d e^{-\frac{a^2}{2x}}da=
\begin{cases}
(d-1)!!x^{\frac{d}{2}},&\text{if $d$ is even},\\
0,&\text{if $d$ is odd}.
\end{cases}
\end{gather}
Therefore, the transformation
\begin{gather}\label{eq:Laplace transformation}
\frac{1}{\prod\sqrt{2\pi x_i}}\int_{\mbR^n}e^{-\sum\frac{a_i^2}{2x_i}}P_n(\sqrt{-1}a;x)da
\end{gather}
just means that we replace each monomial $a_1^{d_1}\cdots a_n^{d_n}$ in $P_n(a;x)$ by 
$$
(-1)^{\frac{1}{2}\sum d_i}\prod\left((d_i-1)!!x_i^{\frac{d_i}{2}}\right),
$$
if all exponents $d_i$ are even and by zero otherwise. The transformation~\eqref{eq:Laplace transformation} can also be interpreted as the Laplace transformation. For $n\ge 3$, in \cite[Remark 1.6]{BSSZ15} it was shown that the function $\frac{P_n(a;x)}{x_1+\cdots+x_n}$ is a power series in $x_1,\ldots,x_n$ with coefficients from~$\mbC[a_1,\ldots,a_n]$. Therefore, the right-hand side of~\eqref{eq:main formula} is a power series in~$x_1,\ldots,x_n$.  In the case $n=2$, again from~\cite[Remark 1.6]{BSSZ15} we know that $\frac{P_2(a_1,a_2;x_1,x_2)}{x_1+x_2}-\frac{1}{x_1+x_2}$ is a power series in~$x_1,x_2$ with coefficients from $\mbC[a_1,a_2]$. This implies that the right-hand side of~\eqref{eq:main formula} is a power series in~$x_1$ and~$x_2$. For $n=1$ we immediately see that the right-hand side of~\eqref{eq:main formula} is equal to $\frac{e^{\frac{x_1^3}{24}}-1}{x_1^2}$ that is of course a power series in~$x_1$.

Let us describe briefly the plan of the proof of Theorem~\ref{theorem:main}. The double ramification cycle~$\DR_g(a_1,\ldots,a_n)$ is a cohomology class in $H^{2g}(\oM_{g,n},\mbQ)$. It depends on a list of integers $a_1,\ldots,a_n$ satisfying $\sum a_i=0$. In~\cite{BSSZ15} the authors derived an explicit formula for the generating series 
$$
\sum_{\substack{g\ge 0\\2g-2+n>0}}\sum_{d_1,\ldots,d_n\ge 0}\left(\int_{\oM_{g,n}}\DR_g(a_1,\ldots,a_n)\psi_1^{d_1}\cdots\psi_n^{d_n}\right)x_1^{d_1}\cdots x_n^{d_n}.
$$   
We recall it in Section~\ref{section:double ramification cycles}. Denote by $\pi_m\colon\oM_{g,n}\to\oM_{g,n-m}$ the forgetful map that forgets the last~$m$ marked points. Let $a_1,\ldots,a_n$ and $b_1,\ldots,b_g$ be arbitrary integers and let $A:=\sum a_i$ and $B:=\sum b_j$. The push-forward $\pi_{g*}\DR_g(a_1,\ldots,a_n,-A-B,b_1,\ldots,b_g)$ is a cohomology class of degree~$0$ and it is equal to $g!b_1^2b_2^2\cdots b_g^2$ times the unit. This observation together with the string equation for the intersection numbers~\eqref{eq:intersection numbers} implies that
\begin{multline*}
g!b_1^2\cdots b_g^2\left(\sum x_i\right)\mcF_g(x_1,\ldots,x_n)=\\
=\sum_{d_1,\ldots,d_n\ge 0}\left(\int_{\oM_{g,n+g+1}}\pi_{g*}\DR_g(a_1,\ldots,a_n,-A-B,b_1,\ldots,b_g)\psi_1^{d_1}\cdots\psi_n^{d_n}\right)x_1^{d_1}\cdots x_n^{d_n}.
\end{multline*}
Moreover, this polynomial is the term of the lowest degree $3g-2+n$ with respect to the $x$-variables in the infinite series
\begin{gather*}
\sum_{m\ge 0}\sum_{d_1,\ldots,d_n\ge 0}\left(\int_{\oM_{g,n+g+1}}\pi_{g*}\DR_m(a_1,\ldots,a_n,-A-B,b_1,\ldots,b_g)\psi_1^{d_1}\cdots\psi_n^{d_n}\right)x_1^{d_1}\cdots x_n^{d_n}.
\end{gather*}
In Section~\ref{section:double ramification cycles} we derive an explicit formula for this series and then in Section~\ref{section:proof of the main theorem} prove Theorem~\ref{theorem:main}.

In Section~\ref{section:examples} we do an explicit computation of the one-point and the two-point functions using our general formula.


\section{Double ramification cycles}\label{section:double ramification cycles}

Let $n\ge 2$ and let $a_1,\ldots,a_n$ be a list of integers satisfying $\sum a_i=0$. The double ramification cycle $\DR_g(a_1,\ldots,a_n)$ is a cohomology class in $H^{2g}(\oM_{g,n},\mbQ)$. If not all of $a_i$'s are equal to zero, then the restriction
$$
\left.\DR_g(a_1,\ldots,a_n)\right|_{\cM_{g,n}}
$$
to the moduli space of smooth curves $\cM_{g,n}\subset\oM_{g,n}$ can be defined as the Poincar\'e dual to the locus of pointed smooth curves~$[C,p_1,\ldots,p_n]$ satisfying $\mathcal O_C\left(\sum a_ip_i\right)\cong\mathcal O_C$. We refer the reader, for example, to~\cite{BSSZ15} for the general definition of the double ramification cycle on the whole moduli space~$\oM_{g,n}$. We will often consider the Poincar\'e dual to the double ramification cycle~$\DR_g(a_1,\ldots,a_n)$. It is an element of $H_{2(2g-3+n)}(\oM_{g,n},\mbQ)$ and, abusing our notation a little bit, it will also be denoted by $\DR_g(a_1,\ldots,a_n)$. An explicit formula for the double ramification cycle~$\DR_g(a_1,\ldots,a_n)$ in terms of tautological classes in the cohomology~$H^*(\oM_{g,n},\mbQ)$ was recently obtained in~\cite{JPPZ16}. In particular, it implies that the double ramification cycle~$\DR_g(a_1,\ldots,a_n)$ depends polynomially on the parameters~$a_1,\ldots,a_n$.  

Let us now formulate the main result of~\cite{BSSZ15}. Let $d_1,\ldots,d_n$ be non-negative integers satisfying $\sum d_i=2g-3+n$. Then the integral
\begin{gather}\label{eq:DR integral}
\int_{\DR_g(a_1,\ldots,a_n)}\psi_1^{d_1}\cdots\psi_n^{d_n}
\end{gather}
is equal to the coefficient of $x_1^{d_1}\cdots x_n^{d_n}$ in the generating function
\begin{gather}\label{eq:DR generating series}
\frac{P_n(a_1,\ldots,a_n;x_1,\ldots,x_n)}{\zeta(x_1+\cdots+x_n)}-\frac{\delta_{n,2}}{x_1+x_2}.
\end{gather}
As it was discussed in the introduction, this generating series is a power series in $x_1,\ldots,x_n$ with coefficients from $\mbC[a_1,\ldots,a_n]$. We see that the integral~\eqref{eq:DR integral} is a polynomial in $a_1,\ldots,a_n$. Note that $\zeta(x)=\sum_{i\ge 0}\frac{x^{2i+1}}{2^{2i}(2i+1)!}$, then from the definition~\eqref{eq:main definition} it is easy to see that the series~\eqref{eq:DR generating series} has the form
$$
\frac{P_n(a_1,\ldots,a_n;x_1,\ldots,x_n)}{\zeta(x_1+\cdots+x_n)}-\frac{\delta_{n,2}}{x_1+x_2}=\sum_{\substack{j\ge 0\\2j-2+n>0}}\sum_i f_{i,j}(a)g_{i,j}(x),
$$
where $g_{i,j}(x)\in\mbC[x_1,\ldots,x_n]$ is a polynomial of degree $2j-3+n$, $f_{i,j}(a)\in\mbC[a_1,\ldots,a_n]$ and for any fixed~$j$ the second summation is finite. From this we conclude that all terms in the series~\eqref{eq:DR generating series} have the geometrical meaning and, thus,
\begin{gather}\label{eq:generating series of DR integrals}
\sum_{\substack{g\ge 0\\2g-2+n>0}}\sum_{d_1,\ldots,d_n\ge 0}\left(\int_{\DR_g(a_1,\ldots,a_n)}\psi_1^{d_1}\cdots\psi_n^{d_n}\right)x_1^{d_1}\cdots x_n^{d_n}=\frac{P(a_1,\ldots,a_n;x_1,\ldots,x_n)}{\zeta(x_1+\cdots+x_n)}-\frac{\delta_{n,2}}{x_1+x_2}.
\end{gather}

Now we want to generalize formula~\eqref{eq:generating series of DR integrals} for the integrals over the double ramification cycles with forgotten points. We begin with the following lemma.

\begin{lemma}\label{lemma:setting 0}
For $n\ge 2$ we have
$$
\left.P_n(a_1,\ldots,a_n;x_1,\ldots,x_n)\right|_{x_n=0}=\frac{\zeta(a_n(x_1+\cdots+x_{n-1}))}{a_n}P_{n-1}(a_1,\ldots,a_{n-1};x_1,\ldots,x_{n-1}).
$$
\end{lemma}
\begin{proof}
For $n=2$ we have
$$
P_2(a_1,a_2;x_1,x_2)=\frac{\zeta(a_1x_2-a_2x_1)}{a_1x_2-a_2x_1}.
$$
Therefore,
$$
\left.P_2(a_1,a_2;x_1,x_2)\right|_{x_2=0}=\frac{\zeta(a_2x_1)}{a_2x_1}=\frac{\zeta(a_2x_1)}{a_2}P_1(a_1;x_1).
$$

Suppose $n\ge 3$. Note that if we set $x_n=0$, then the product $x_2'\cdots x_{n-1}'$ on the right-hand side of~\eqref{eq:main definition} vanishes unless $\sigma(n)=n$. Therefore, we obtain
\begin{align*}
&\left.P_n(a_1,\ldots,a_n;x_1,\ldots,x_n)\right|_{x_n=0}=\sum_{\substack{\sigma\in S_n\\\sigma(1) = 1,\,\sigma(n)=n}}
x'_2\cdots x'_{n-1}\times\\
&\hspace{1cm}\times
\frac{
\zeta\left(
\begin{array}{|cc|}
a_1' & a_2'
\\ x_1' & x_2'
\end{array}
\right)
\cdots
\zeta\left(
\begin{array}{|cc|}
a_1'+a_2'+\cdots+a_{n-2}' & a_{n-1}'
\\ x_1'+x_2'+\cdots+x_{n-2}' & x_{n-1}'
\end{array}
\right)
\zeta\left(
\begin{array}{|cc|}
a_1' + \cdots + a_{n-1}' & a_n
\\ x_1' + \cdots + x_{n-1}' & 0
\end{array}
\right)}
{\begin{array}{|cc|}
a_1' & a_2'
\\ x_1' & x_2'
\end{array}
\;
\cdots 
\;
\begin{array}{|cc|}
a_{n-2}' & a_{n-1}'
\\ x_{n-2}' & x_{n-1}'
\end{array}
\;\;
\begin{array}{|cc|}
a_{n-1}' & a_n
\\ x_{n-1}' & 0
\end{array}
}=\\
=&\frac{\zeta(a_n(x_1+\cdots+x_{n-1}))}{a_n}\sum_{\substack{\sigma\in S_{n-1}\\\sigma(1) = 1}}
x'_2\ldots x'_{n-2}
\frac{
\zeta\left(
\begin{array}{|cc|}
a_1' & a_2'
\\ x_1' & x_2'
\end{array}
\right)
\cdots
\zeta\left(
\begin{array}{|cc|}
a_1'+a_2'+\cdots+a_{n-2}' & a_{n-1}'
\\ x_1'+x_2'+\cdots+x_{n-2}' & x_{n-1}'
\end{array}
\right)
}
{\begin{array}{|cc|}
a_1' & a_2'
\\ x_1' & x_2'
\end{array}
\;
\cdots 
\;
\begin{array}{|cc|}
a_{n-2}' & a_{n-1}'
\\ x_{n-2}' & x_{n-1}'
\end{array}
}=\\
=&\frac{\zeta(a_n(x_1+\cdots+x_{n-1}))}{a_n}P_{n-1}(a_1,\ldots,a_{n-1};x_1,\ldots,x_{n-1}).
\end{align*}
The lemma is proved.
\end{proof}

Introduce the series $S(x):=\frac{\zeta(x)}{x}=1+\sum_{i\ge 1}\frac{x^{2i}}{2^{2i}(2i+1)!}$.
\begin{lemma}\label{lemma:generating series with zeroes}
Let $n\ge 3$, $m\ge 0$ and $a_1,\ldots,a_n,b_1,\ldots,b_m$ be integers satisfying $\sum a_i+\sum b_j=0$. Then we have
\begin{multline*}
\sum_{g\ge 0}\sum_{d_1,\ldots,d_n\ge 0}\left(\int_{\DR_g(a_1,\ldots,a_n,b_1,\ldots,b_m)}\psi_1^{d_1}\cdots\psi_n^{d_n}\right)x_1^{d_1}\cdots x_n^{d_n}=\\
=\frac{\prod_{i=1}^m S(b_i X)}{S(X)}X^{m-1}P_n(a_1,\ldots,a_n;x_1,\ldots,x_n),
\end{multline*}
where $X:=\sum x_i$.
\end{lemma}
\begin{proof}
Using equation~\eqref{eq:generating series of DR integrals} and Lemma~\ref{lemma:setting 0}, we compute
\begin{align*}
\sum_{g\ge 0}\sum_{d_1,\ldots,d_n\ge 0}&\left(\int_{\DR_g(a_1,\ldots,a_n,b_1,\ldots,b_m)}\psi_1^{d_1}\cdots\psi_n^{d_n}\right)x_1^{d_1}\cdots x_n^{d_n}=\\
=&\left.\frac{P_{n+m}(a_1,\ldots,a_n,b_1,\ldots,b_m;x_1,\ldots,x_{n+m})}{\zeta(x_1+\cdots+x_{n+m})}\right|_{x_{n+1}=\cdots=x_{n+m}=0}=\\
=&\frac{1}{\zeta(X)}\left(\prod_{i=1}^m\frac{\zeta(b_iX)}{b_i}\right)P_n(a_1,\ldots,a_n;x_1,\ldots,x_n)=\\
=&\frac{\prod_{i=1}^m S(b_i X)}{S(X)}X^{m-1}P_n(a_1,\ldots,a_n;x_1,\ldots,x_n).
\end{align*}
The lemma is proved.
\end{proof}

We keep the assumptions of the last lemma. Recall that by $\pi_m\colon\oM_{g,n+m}\to\oM_{g,n}$ we denote the forgetful map that forgets the last~$m$ marked points. For a subset $I\subset\{1,\ldots,m\}$ let
$$
B_I:=\sum_{i\in I}b_i.
$$
The following proposition generalizes formula~\eqref{eq:generating series of DR integrals}.
\begin{proposition}\label{proposition:series with forgotten points}
We have
\begin{multline}\label{eq:generating series with m forgotten points}
\sum_{g\ge 0}\sum_{d_1,\ldots,d_n\ge 0}\left(\int_{\pi_{m*}\DR_g(a_1,\ldots,a_n,b_1,\ldots,b_m)}\psi_1^{d_1}\cdots\psi_n^{d_n}\right)x_1^{d_1}\cdots x_n^{d_n}=\\
=\frac{1}{S(X)}\sum_{I_0\sqcup I_1\sqcup\ldots\sqcup I_n=\{1,\ldots,m\}}X^{|I_0|-1}\prod_{i=1}^n(-x_i)^{|I_i|}\prod_{i\in I_0}S(b_iX)P_n\left(a_1+B_{I_1},\ldots,a_n+B_{I_n};x\right).
\end{multline}
\end{proposition}
\begin{proof}
For a subset $I=\{i_1,i_2,\ldots,i_{|I|}\}\subset\{1,\ldots,m\}$, where $i_1<i_2<\cdots<i_{|I|}$, we denote by~$\ob_I$ the string $b_{i_1},b_{i_2},\ldots,b_{i_{|I|}}$. Clearly, formula~\eqref{eq:generating series with m forgotten points} follows from Lemma~\ref{lemma:generating series with zeroes} and the equation
\begin{gather}\label{eq:DR integral with forgotten points}
\int_{\pi_{m*}\DR_g(a_1,\ldots,a_n,b_1,\ldots,b_m)}\psi_1^{d_1}\cdots\psi_n^{d_n}=\sum_{\substack{\coprod_{i=0}^n I_i=\{1,\ldots,m\}\\|I_i|\le d_i}}(-1)^{m-|I_0|}\int_{\DR_g\left(a_1+B_{I_1},\ldots,a_n+B_{I_n},\ob_{I_0}\right)}\prod_{i=1}^n\psi_i^{d_i-|I_i|}.
\end{gather}
The proof of this equation is very similar to the proof of Proposition~2.3 in~\cite{BS11}. Let us first prove that for $m\ge 1$ we have

\begin{align}
\int_{\pi_{m*}\DR_g(a_1,\ldots,a_n,b_1,\ldots,b_m)}\psi_1^{d_1}\cdots\psi_n^{d_n}=&\int_{(\pi_{m-1})_*\DR_g(a_1,\ldots,a_n,b_m,b_1,\ldots,b_{m-1})}\psi_1^{d_1}\cdots\psi_n^{d_n}\label{eq:remembering one point}\\
&-\sum_{\substack{1\le i\le n\\d_i>0}}\int_{(\pi_{m-1})_*\DR_g(a_1,\ldots,a_i+b_m,\ldots,a_n,b_1,\ldots,b_{m-1})}\psi_i^{d_i-1}\prod_{j\ne i}\psi_j^{d_j}.\notag
\end{align}
For this we write $\pi_m=\pi_1\circ\pi_{m-1}$, where $\pi_{m-1}\colon\oM_{g,n+m}\to\oM_{g,n+1}$ and $\pi_1\colon\oM_{g,n+1}\to\oM_{g,n}$. Then we have
$$
\int_{\pi_{m*}\DR_g(a_1,\ldots,a_n,b_1,\ldots,b_m)}\psi_1^{d_1}\cdots\psi_n^{d_n}=\int_{(\pi_{m-1})_*\DR_g(a_1,\ldots,a_n,b_m,b_1,\ldots,b_{m-1})}\pi_1^*\left(\psi_1^{d_1}\cdots\psi_n^{d_n}\right).
$$
For an integer $N\ge 2$ and a subset $J\in\{1,\ldots,N\}$, $|J|\ge 2$, denote by $\delta_0^J$ the cohomology class in $H^2(\oM_{g,N},\mbQ)$ that is Poincar\'e dual to the divisor whose generic point is a nodal curve made of one smooth component of genus $0$ with the marked points labeled by the set~$J$ and of another smooth component of genus~$g$ with the remaining marked points, joined at a separating node. If $1\le i\le n$ and $d\ge 1$, then we have $\pi_1^*\psi_i^d=\psi_i^d-\delta_0^{\{i,n+1\}}\pi_1^*\psi_i^{d-1}$. Therefore,
\begin{align*}
&\pi_1^*\left(\psi_1^{d_1}\cdots\psi_n^{d_n}\right)=\psi_1^{d_1}\cdots\psi_n^{d_n}-\sum_{\substack{1\le i\le n\\d_i>0}}\delta_0^{\{i,n+1\}}\pi_1^*\left(\psi_i^{d_i-1}\prod_{j\ne i}\psi_j^{d_j}\right).
\end{align*}
We have (see~\cite{BSSZ15} and, in particular, Section 2.1 there for the explanation of the notation~$\boxtimes$)
\begin{multline*}
\delta_0^{\{i,n+1\}}\cdot(\pi_{m-1})_*\DR_g(a_1,\ldots,a_n,b_m,b_1,\ldots,b_{m-1})=\\
=(\pi_{m-1})_*\left(\DR_0(a_i,b_m,-a_i-b_m)\boxtimes\DR_g(a_1,\ldots,\widehat{a_i},\ldots,a_n,b_1,\ldots,b_{m-1},a_i+b_m)\right).
\end{multline*}
Therefore,
\begin{multline*}
\int_{(\pi_{m-1})_*\DR_g(a_1,\ldots,a_n,b_m,b_1,\ldots,b_{m-1})}\delta_0^{\{i,n+1\}}\pi_1^*\left(\psi_i^{d_i-1}\prod_{j\ne i}\psi_j^{d_j}\right)=\\
=\int_{(\pi_{m-1})_*\DR_g(a_1,\ldots,a_i+b_m,\ldots,a_n,b_1,\ldots,b_{m-1})}\psi_i^{d_i-1}\prod_{j\ne i}\psi_j^{d_j}.
\end{multline*}
This completes the proof of equation~\eqref{eq:remembering one point}.

It is easy to see that, applying formula~\eqref{eq:remembering one point} $m$ times, we get equation~\eqref{eq:DR integral with forgotten points}. The proposition is proved.
\end{proof}

Finally, let us formulate an important geometric property of the double ramification cycle that we will use in the next section. Let $g\ge 0$, $n\ge 3$ and $a_1,\ldots,a_n,b_1,\ldots,b_g$ be integers satisfying $\sum a_i+\sum b_j=0$. Then we have (see \cite[Example 3.7]{BSSZ15})
\begin{align}
&\pi_{g*}\DR_g(a_1,\ldots,a_n,b_1,\ldots,b_g)=g!b_1^2\cdots b_g^2[\oM_{g,n}],&&\label{eq:DR and fundamental}\\
&\pi_{g*}\DR_m(a_1,\ldots,a_n,b_1,\ldots,b_g)=0,&& m<g.\label{eq:DR and zero}
\end{align}
The last property implies that the coefficient of a monomial $x_1^{d_1}\cdots x_n^{d_n}$ in the generating series~\eqref{eq:generating series with m forgotten points} is zero, if $\sum d_i<3m-3+n$. We will use this observation in the next section.


\section{Proof of Theorem~\ref{theorem:main}}\label{section:proof of the main theorem}

We divide this section in two parts. In Section~\ref{subsection:reduction} we use the string equation for the intersection numbers~\eqref{eq:intersection numbers} in order to reduce Theorem~\ref{theorem:main} to the case~$n\ge 3$. We treat this case in Section~\ref{subsection:main case}.

\subsection{Reduction to the case $n\ge 3$}\label{subsection:reduction}

The intersection numbers $\<\tau_{d_1}\cdots\tau_{d_n}\>_g$ satisfy the string equation
$$
\<\tau_0\tau_{d_1}\cdots\tau_{d_n}\>_g=\sum_{\substack{1\le i\le n\\d_i>0}}\<\tau_{d_i-1}\prod_{j\ne i}\tau_{d_j}\>_g,
$$
with one exceptional case $\<\tau_0^3\>_0=1$. The string equation implies that for $n\ge 2$ we have
\begin{gather}\label{eq:string for n-point}
\left.\mcF(x_1,\ldots,x_n)\right|_{x_n=0}=(x_1+\cdots+x_{n-1})\mcF(x_1,\ldots,x_{n-1})+\delta_{n,3}.
\end{gather}
Denote the right-hand side of~\eqref{eq:main formula} by $\mcG(x_1,\ldots,x_n)$. Let us show that the function $\mcG$ also satisfies equation~\eqref{eq:string for n-point}. We compute
\begin{align*}
&\left.\frac{1}{\prod_{i=1}^n\sqrt{2\pi x_i}}\int_{\mbR^n}e^{-\sum_{i=1}^n\frac{a_i^2}{2x_i}}P_n\left(\sqrt{-1}a_1,\ldots,\sqrt{-1}a_n;x_1,\ldots,x_n\right)da_1\cdots da_n\right|_{x_n=0}=\\
=&\left.\frac{1}{(2\pi)^{\frac{n}{2}}}\int_{\mbR^n}e^{-\sum_{i=1}^n\frac{a_i^2}{2}}P_n\left(\sqrt{-x_1}a_1,\ldots,\sqrt{-x_n}a_n;x_1,\ldots,x_n\right)da_1\cdots da_n\right|_{x_n=0}\stackrel{\text{by Lemma~\ref{lemma:setting 0}}}{=}\\
=&\frac{x_1+\cdots+x_{n-1}}{(2\pi)^{\frac{n}{2}}}\int_{\mbR^n}e^{-\sum_{i=1}^n\frac{a_i^2}{2}}P_{n-1}\left(\sqrt{-x_1}a_1,\ldots,\sqrt{-x_{n-1}}a_{n-1};x_1,\ldots,x_{n-1}\right)da_1\cdots da_n=\\
=&\frac{x_1+\cdots+x_{n-1}}{\prod_{i=1}^{n-1}\sqrt{2\pi x_i}}\int_{\mbR^{n-1}}e^{-\sum_{i=1}^{n-1}\frac{a_i^2}{2x_i}}P_{n-1}\left(\sqrt{-1}a_1,\ldots,\sqrt{-1}a_{n-1};x_1,\ldots,x_{n-1}\right)da_1\cdots da_{n-1}.
\end{align*}
Thus, we obtain
\begin{gather}\label{eq:string for G}
\left.\mcG(x_1,\ldots,x_n)\right|_{x_n=0}=(x_1+\cdots+x_{n-1})\mcG(x_1,\ldots,x_{n-1})+\delta_{n,3}.
\end{gather}
We see that if equation~\eqref{eq:main formula} holds for $n=n_0$, then it holds for all $n\le n_0$. Hence, it is sufficient to prove equation~\eqref{eq:main formula} for $n\ge 3$.

\subsection{Case $n\ge 3$}\label{subsection:main case}

We assume that $n\ge 3$. Let $g\ge 0$ and consider arbitrary integers $a_1,\ldots,a_n$ and $b_1,\ldots,b_g$. Let 
\begin{gather*}
A:=\sum_{i=1}^n a_i,\qquad B:=\sum_{i=1}^g b_i,\qquad X:=\sum_{i=1}^n x_i.
\end{gather*}
Equations~\eqref{eq:DR and fundamental} and~\eqref{eq:string for n-point} imply that 
\begin{align}
g!b_1^2\cdots b_g^2 X\mcF_g(x_1,\ldots,x_n)=&g!b_1^2\cdots b_g^2\left.\mcF_g(x_1,\ldots,x_n,x_{n+1})\right|_{x_{n+1}=0}=\notag\\
=&\sum_{d_1,\ldots,d_n\ge 0}\left(\int_{\pi_{g*}\DR_g(a_1,\ldots,a_n,-A-B,b_1,\ldots,b_g)}\psi_1^{d_1}\cdots\psi_n^{d_n}\right)x_1^{d_1}\cdots x_n^{d_n}\label{first step}
\end{align}
independently of $a_1,\ldots,a_n$. Moreover, by the remark at the end of Section~\ref{section:double ramification cycles}, expression~\eqref{first step} is equal to the term of the lowest degree~$3g-2+n$ with respect to the $x$-variables in the series
\begin{gather}\label{series}
\sum_{m\ge 0}\sum_{d_1,\ldots,d_n\ge 0}\left(\int_{\pi_{g*}\DR_m(a_1,\ldots,a_n,-A-B,b_1,\ldots,b_g)}\psi_1^{d_1}\cdots\psi_n^{d_n}\right)x_1^{d_1}\cdots x_n^{d_n}.
\end{gather}
Proposition~\ref{proposition:series with forgotten points} and Lemma~\ref{lemma:setting 0} imply that
\begin{align}\label{eq:generating series with forgotten points}
&\sum_{m\ge 0}\sum_{d_1,\ldots,d_n\ge 0}\left(\int_{\pi_{g*}\DR_m(a_1,\ldots,a_n,-A-B,b_1,\ldots,b_g)}\psi_1^{d_1}\cdots\psi_n^{d_n}\right)x_1^{d_1}\cdots x_n^{d_n}=\\
=&\frac{S((-A-B)X)}{S(X)}\sum_{\coprod_{i=0}^n I_i=\{1,\ldots,g\}}X^{|I_0|}\prod_{i=1}^n(-x_i)^{|I_i|}\prod_{i\in I_0}S(b_iX)P_n\left(a_1+B_{I_1},\ldots,a_n+B_{I_n};x\right).\notag
\end{align}
In the last expression the multiplication by $\frac{S((-A-B)X)}{S(X)}$ doesn't change the lowest degree term with respect to the $x$-variables. Therefore, we obtain
\begin{align*}
g!b_1^2\cdots b_g^2&X\mcF_g(x_1,\ldots,x_n)=\\
=&\left[\sum_{\coprod_{i=0}^n I_i=\{1,\ldots,g\}}X^{|I_0|}\prod_{i=1}^n(-x_i)^{|I_i|}\prod_{i\in I_0}S(b_iX)P_n\left(a_1+B_{I_1},\ldots,a_n+B_{I_n};x\right)\right]_{3g-2+n}=\\
=&\left[\sum_{\coprod_{i=0}^nI_i=\{1,\ldots,g\}}X^{|I_0|}\prod_{i=1}^n(-x_i)^{|I_i|}\prod_{i\in I_0}S(b_iX)P_n\left(B_{I_1},\ldots,B_{I_n};x\right)\right]_{3g-2+n},
\end{align*}
where by $[\cdot]_d$ we denote the degree $d$ part with respect to the $x$-variables. Equivalently, we can write
\begin{multline}\label{second step}
g!X\mcF_g(x_1,\ldots,x_n)=\\
=\left[\Coef_{b_1^2\cdots b_g^2}\sum_{\coprod_{i=0}^n I_i=\{1,\ldots,g\}}X^{|I_0|}\prod_{i=1}^n(-x_i)^{|I_i|}\prod_{i\in I_0}S(b_iX)P_n\left(B_{I_1},\ldots,B_{I_n};x\right)\right]_{3g-2+n}.
\end{multline}
From the definition~\eqref{eq:main definition} we see that the series~$P_n(a_1,\ldots,a_n;x_1,\ldots,x_n)$ satisfies the following homogeneity property:
$$
P_n(\lambda^{-1}a_1,\ldots,\lambda^{-1}a_n;\lambda x_1,\ldots,\lambda x_n)=\lambda^{n-2}P_n(a_1,\ldots,a_n;x_1,\ldots,x_n).
$$
Therefore, the expression in the square brackets on the right-hand side of~\eqref{second step} has automatically degree~$3g-2+n$ in the $x$-variables. Since $S(z)=1+\frac{z^2}{24}+\ldots$, we have
$$
\Coef_{b_1^2\cdots b_g^2}\left(\prod_{i\in I_0}S(b_iX)P_n\left(B_{I_1},\ldots,B_{I_n};x\right)\right)=\frac{X^{2|I_0|}}{24^{|I_0|}}\Coef_{\prod_{i\notin I_0}b_i^2}P_n\left(B_{I_1},\ldots,B_{I_n};x\right).
$$
Note that
$$
\Coef_{\prod_{i\notin I_0}b_i^2}\prod_{i=1}^n B_{I_i}^{d_i}=
\begin{cases}
\prod_{i=1}^n\frac{d_i!}{2^{d_i/2}},&\text{if $d_i=2|I_i|$},\\
0,&\text{otherwise}.
\end{cases}
$$
Therefore, we get
\begin{align*}
\Coef_{b_1^2\cdots b_g^2}&\sum_{\coprod_{i=0}^nI_i=\{1,\ldots,g\}}X^{|I_0|}\prod_{i=1}^n(-x_i)^{|I_i|}\prod_{i\in I_0}S(b_iX)P_n\left(B_{I_1},\ldots,B_{I_n};x\right)=\\
=&\sum_{\coprod_{i=0}^nI_i=\{1,\ldots,g\}}\frac{X^{3|I_0|}}{24^{|I_0|}}\prod_{i=1}^n\left((-x_i)^{|I_i|}\frac{(2|I_i|)!}{2^{|I_i|}}\right)\Coef_{\prod_{i=1}^na_i^{2|I_i|}}P_n(a;x)=\\
=&\sum_{\substack{m_0,m_1,\ldots,m_n\ge 0\\m_0+\cdots+m_n=g}}\frac{g!}{m_0!\cdots m_n!}\left(\frac{X^3}{24}\right)^{m_0}\prod_{i=1}^n\left((-x_i)^{m_i}\frac{(2m_i)!}{2^{m_i}}\right)\Coef_{\prod_{i=1}^na_i^{2m_i}}P_n(a;x)=\\
=&\sum_{\substack{m_0,m_1,\ldots,m_n\ge 0\\m_0+\cdots+m_n=g}}\frac{g!}{m_0!}\left(\frac{X^3}{24}\right)^{m_0}\prod_{i=1}^n\left((-x_i)^{m_i}(2m_i-1)!!\right)\Coef_{\prod_{i=1}^na_i^{2m_i}}P_n(a;x).
\end{align*}
Finally, we obtain
\begin{align*}
X\mcF(x_1,\ldots,x_n)=&\sum_{g\ge 0}X\mcF_g(x_1,\ldots,x_n)=\\
=&\sum_{m_0,m_1,\ldots,m_n\ge 0}\frac{1}{m_0!}\left(\frac{X^3}{24}\right)^{m_0}\prod_{i=1}^n\left((-x_i)^{m_i}(2m_i-1)!!\right)\Coef_{\prod_{i=1}^na_i^{2m_i}}P_n(a;x)\stackrel{\text{by~\eqref{eq:Gauss integral}}}{=}\\
=&\frac{e^{\frac{X^3}{24}}}{\prod\sqrt{2\pi x_i}}\int_{\mbR^n}e^{-\sum\frac{a_i^2}{2x_i}}P_n\left(\sqrt{-1}a;x\right)da.
\end{align*}
This completes the proof of the theorem.


\section{Examples}\label{section:examples}

For $n=1$ Theorem~\ref{theorem:main} immediately gives the well-known formula (see e.g.~\cite{FP00})
$$
\mcF(x_1)=\frac{e^{\frac{x_1^3}{24}}-1}{x_1^2}.
$$

Suppose $n=2$. We have 
\begin{align*}
P_2(a_1,a_2;x_1,x_2)=&\frac{\zeta(a_1x_2-a_2x_1)}{a_1x_2-a_2x_1}=S(a_1x_2-a_2x_1)=\sum_{n\ge 0}\frac{(a_1x_2-a_2x_1)^{2n}}{(2n+1)!2^{2n}}=\\
=&\sum_{n\ge 0}\sum_{m_1+m_2=2n}\frac{(a_1x_2)^{m_1}(-a_2x_1)^{m_2}}{(2n+1)2^{2n}m_1!m_2!}.
\end{align*}
Therefore, we obtain
\begin{align*}
\frac{e^{\frac{(x_1+x_2)^3}{24}}}{2\pi\sqrt{x_1x_2}}\int_{\mbR^2}&e^{-\frac{a_1^2}{2x_1}-\frac{a_2^2}{2x_2}}P_2(\sqrt{-1}a_1,\sqrt{-1}a_2;x_1,x_2)da=\\
=&\frac{e^{\frac{(x_1+x_2)^3}{24}}}{2\pi\sqrt{x_1x_2}}\sum_{n\ge 0}\sum_{m_1+m_2=n}\int_{\mbR^2}e^{-\frac{a_1^2}{2x_1}-\frac{a_2^2}{2x_2}}\left(\frac{(-1)^na_2^{2m_1}a_1^{2m_2}x_1^{2m_1}x_2^{2m_2}}{2^{2n}(2n+1)(2m_1)!(2m_2)!}\right)da=\\
=&e^{\frac{(x_1+x_2)^3}{24}}\sum_{n\ge 0}\sum_{m_1+m_2=n}\frac{(-1)^nx_2^{m_1}x_1^{m_2}x_1^{2m_1}x_2^{2m_2}}{2^{3n}(2n+1)m_1!m_2!}=\\
=&e^{\frac{x_1^3+x_2^3}{24}}e^{\frac{x_1x_2(x_1+x_2)}{8}}\sum_{n\ge 0}\frac{(-1)^n(x_1x_2(x_1+x_2))^n}{8^n(2n+1)n!}=\\
=&e^{\frac{x_1^3+x_2^3}{24}}\sum_{n\ge 0}\left(\frac{x_1x_2(x_1+x_2)}{8}\right)^n\sum_{m_1+m_2=n}\frac{(-1)^{m_2}}{m_1!m_2!(2m_2+1)}.
\end{align*}
Let $C_n:=\sum_{m_1+m_2=n}\frac{(-1)^{m_2}}{m_1!m_2!(2m_2+1)}$. For $n\ge 1$ we compute 
$$
C_n=\frac{1}{n!}\int_0^1(1-y^2)^ndy=\frac{2}{(n-1)!}\int_0^1y^2(1-y^2)^{n-1}dy=-2nC_n+2C_{n-1}.
$$
Therefore, $C_n=\frac{2}{2n+1}C_{n-1}$, and, since $C_0=1$, we get $C_n=\frac{2^n}{(2n+1)!!}$. Hence, we obtain
\begin{gather*}
\frac{e^{\frac{(x_1+x_2)^3}{24}}}{2\pi\sqrt{x_1x_2}}\int_{\mbR^2}e^{-\frac{a_1^2}{2x_1}-\frac{a_2^2}{2x_2}}P_2(\sqrt{-1}a_1,\sqrt{-1}a_2;x_1,x_2)da=e^{\frac{x_1^3+x_2^3}{24}}\sum_{n\ge 0}\frac{n!}{(2n+1)!}\left(\frac{x_1x_2(x_1+x_2)}{2}\right)^n.
\end{gather*}
As a result, we get the following formula for the two-point function:
\begin{align*}
\mcF(x_1,x_2)=\frac{e^{\frac{x_1^3+x_2^3}{24}}}{(x_1+x_2)}\sum_{n\ge 0}\frac{n!}{(2n+1)!}\left(\frac{x_1x_2(x_1+x_2)}{2}\right)^n-\frac{1}{x_1+x_2}.
\end{align*}
This formula was found by R.~Dijkgraaf (see e.g.~\cite{FP00}).

\end{document}